\documentclass[12pt]{amsart}

\usepackage[colorlinks=true,linkcolor=blue]{hyperref}
\usepackage{amsmath}
\usepackage{amssymb}
\usepackage{amsthm}
\usepackage{enumerate}

\theoremstyle{plain} 
\newtheorem{theorem}{Theorem}[section]
\newtheorem{proposition}[theorem]{Proposition}
\newtheorem{lemma}[theorem]{Lemma}
\newtheorem{corollary}[theorem]{Corollary} 

\newtheorem{conjecture}[theorem]{Conjecture}

\theoremstyle{remark}
\newtheorem{remark}[theorem]{Remark}

\newtheorem{example}[theorem]{Example}

\theoremstyle{definition}

\DeclareMathOperator{\Gal}{Gal}

\DeclareMathOperator{\Frob}{Frob}

\DeclareMathOperator{\FPP}{FPP}

\DeclareMathOperator{\Aut}{Aut}

\newcommand{\fp}{ {\mathfrak p} }

\newcommand{\fq}{ {\mathfrak q} }

\newcommand{\fP}{ {\mathfrak P} }
\newcommand{\fQ}{ {\mathfrak Q} }

\newcommand{\cO}{ {\mathcal O} }

\newcommand{\cH}{ {\mathcal H} }

\newcommand{\bZ} { {\mathbb Z}} 
\newcommand{\bN} { {\mathbb N}} 
 
\newcommand{\bF} { {\mathbb F}} 
\newcommand{\bP} { {\mathbb P}} 
\newcommand{\bQ} { {\mathbb Q}} 
\newcommand{\bC} { {\mathbb C}}

\usepackage{tikz}

\begin{document}

\title{Backward Orbits of Critical Points}

\author[J. Juul]{Jamie Juul}
\address{Jamie Juul \\ Department of Mathematics \\ Colorado State University \\ Fort Collins, CO \\ United States}
\email{jamie.juul@colostate.edu}



\begin{abstract} We examine the Galois groups of the extensions $K((f'\circ f^n)^{-1}(0))/K$ where $K$ is a number field for polynomials $f(x)\in K[x]$. We use our understanding of this group to study the proportion of primes for which $f$ has a $\mathfrak{p}$-adic attracting periodic point for a ``typical'' $f$ and apply the result to the split case of the Dynamical Mordell-Lang Conjecture.
\end{abstract}

\maketitle

\section{Introduction}

Let $K$ be a number field, $d\in \bZ_{\geq 2}$, and let $s_1,\dots, s_{d-1}, t$ independent transcendentals over $K$. We are using $t$ for the constant term to differentiate it from the other coefficients as it will play a special role in the proof of the main theorem. Then $f(x)=x^d+s_1x^{d-1}+\dots+s_{d-1}x+t$ is the generic monic polynomial over $K$. Define $K_s:=K(s_1,\dots,s_{d-1})$ and $K_{s,t}:=K_s(t)$. We investigate the action of the absolute Galois group of $K_{s,t}$ on the backward orbits of the critical points of $f(x)$. Giving the elements of the backward orbits the structure of a rooted tree induces a homomorphism from the Galois group to the automorphism group of the tree, called an arboreal Galois representation. In this case, the automorphism group of the tree is isomorphic to $S_{d-1}[[S_d]^n]$. We prove that the action of the Galois group is as large as possible subject to this constraint.
\begin{theorem}\label{thm:main}
	$\Gal\left(K_{s,t}\left(\left(f'\circ f^n\right)^{-1}(0)\right)/K_{s,t}\right)\cong S_{d-1}[[S_d]^n]$.
\end{theorem}

Note, our set up here differs from iterated monodromy groups where one considers preimages of a transcendental basepoint for a specific polynomial or rational function. 

Since $f'(x)$ is itself a polynomial with independent transcendental coefficients, we have $\Gal\left(K_{s,t}\left((f')^{-1}(0)\right)/K_{s,t}\right)\cong S_{d-1}$ by a classical result. Thus, Theorem \ref{thm:main} is equivalent to the following proposition.
\begin{proposition}\label{proposition:main}
$\Gal\left(K_{s,t}\left(\bigcup_{f'(\beta)=0}f^{-n}(\beta)\right)/K_{s,t}((f')^{-1}(0))\right)\cong \prod_{d-1}[S_d]^n$.
\end{proposition}

By Hilbert's irreducibility theorem, the corresponding Galois groups for ``most'' specializations of $f(x)$, that is specializations of the transcendentals to the number field $K$, will be isomorphic to the groups in Theorem \ref{thm:main} and Proposition \ref{proposition:main} (see Corollaries \ref{corollary:specialization1} and \ref{corollary:specialization2}). The Chebotarev Density theorem allows us to translate information about these Galois groups to information about certain sets of primes. Specifically, we use Proposition \ref{proposition:main} to show for a ``typical'' polynomial defined over a number field $K$ and with critical points also defined over $K$ the density of primes for which the polynomial has an attracting point in $\bar{K_\fp}$ is small. Here $\bar{K_\fp}$ denotes the algebraic closure of the completion of $K$ at the prime $\fp$.

\begin{theorem}\label{thm:zariskisetforonemap} Fix a number field $K$ and $\epsilon>0$. There is a Zariski dense subset $\cH$ of $K^d$ such that for all monic polynomials $f(x)$ with critical points $b_1,\dots, b_{d-1}$ and constant term $c_0$ with $(b_1,\dots,b_{d-1},c_0)\in \cH$, 
the density of the set of primes $\fp$ of $K$ for which $f(x)$ has a finite attracting periodic point in $\bar{K}_\fp$ is less than $\epsilon$.
\end{theorem}

\begin{remark} Using known estimates on fixed point proportions for iterated wreath products, we can conclude for any $f(x)$ as in Theorem \ref{thm:zariskisetforonemap}, the set of such primes is bounded above by $\frac{C_d}{n}$, for some constant $C_d$ depending on $d$. 
\end{remark}

We apply our results toward the split case of the Dynamical Mordell-Lang Conjecture. The Dynamical Mordell-Lang conjecture is a dynamical analogue to the cyclic case of the classical Mordell-Lang Conjecture, proven by Faltings \cite{Faltings}, Vojta \cite{vojta}, and McQuillan \cite{McQuillan}. See \cite{BGT} for a detailed narrative on the Dynamical Mordell-Lang conjecture.

\begin{conjecture}[Dynamical Mordell-Lang Conjecture, \cite{GT}] Let $X$ be a quasiprojective variety defined over $\bC$, $\Phi$ be an endomorphism of $X$, $\alpha \in X(\bC)$, and $V$ a subvariety of $X$. Then the set of $n\in \bN_0$ such that $\Phi^n(\alpha)\in V(\bC)$ is a union of finitely many arithmetic progressions.
\end{conjecture}

Using Theorem \ref{thm:boundonattracting}, we are able to prove the Dynamical Mordell-Lang Conjecture holds for ``most'' split maps which act coordinate-wise by polynomials whose critical points are all defined over the same number field. 

\begin{theorem}\label{thm:mainapplication} Let $K$ be a number field. Fix an integer $g\geq 1$. Let $V$ be a subvariety of $(\bP^1)^g$. Fix $g$ integers $d_1,\dots, d_g \geq 2$. Let $\Phi=(f_1,\dots,f_g)$ where $f_i(x)$ is a monic polynomial of degree $d_i$ with critical points $b_{i,1},\dots, b_{i,d_i-1}\in K$ and constant term $c_{i,0}\in K$. Let $\alpha=(\alpha_1,\dots,\alpha_g)\in (\bP^1(\bar{K}))^g$ such that $\alpha_i\neq \infty$ for each $i$. There are Zariski dense sets $\cH(g,d_i,K)\subseteq K^{d_i}$ such that if $(b_{i,1},\dots,b_{i,d_i-1},c_0)\in \cH(g,d_i,K)$, then the set of $n\in \bN_0$ such that $\Phi^n(\alpha)\in V(\bar{K})$ is a union of finitely many arithmetic progressions.
\end{theorem}

The proof of this theorem relies on a theorem of Benedetto, Ghioca, Kurlberg, Tucker \cite[Theorem 3.4]{BGKT} which gives a sufficient condition for the conclusion of the Dynamical Mordell-Lang Conjecture for split maps in terms of empty intersection of the orbit with the residue classes of attracting periodic points modulo a prime $p$. Their proof proceeds by showing it is possible to find a $p$-adic parametrization of the orbit under these hypotheses. We show our results imply the existence of a suitable prime. The Galois theoretic approach to investigating Dynamical Mordell-Lang used in this paper was suggested by Bell, Ghioca, and Tucker \cite{BGT}. Interestingly, heuristics suggest this approach likely will not work for non split maps in dimension greater than 4 since there should be a positive proportion of maps in the moduli space such that the orbit of a point passes through the ramification locus at every prime \cite{BGH}. We also note a similar result was proved by Fakhruddin \cite{Fakhruddin} for generic degree $d$ endomorphisms of $\bP^n_K$, where for the definition of generic used in Fakhruddin's work, a map being generic implies it is in the complement of a countable union of proper subvarieties of the natural parameter variety.  

Before proving the main results of this paper, we introduce arboreal Galois representations in Section \ref{section:arboreal} and give some preliminary results on ramification in Section \ref{section:preliminaries}. In Section \ref{section:proof} we prove Theorem \ref{thm:main}, Proposition \ref{proposition:main}, and corollaries about specializations. Finally, we explore the application to $\fp$-adic attracting periodic points and the connection to the Dynamical Mordell-Lang Conjecture in Section \ref{section:application}.

\section{Arboreal Galois Representation}\label{section:arboreal}

In order to analyze the structure of $\Gal\left(K_{s,t}\left(\left(f'\circ f^n\right)^{-1}(0)\right)/K_{s,t}\right)$ we consider the arboreal Galois representation of this group. These representations have been studied by a number of authors beginning with the work of Odoni \cite{odoni}. A survey of the subject can be found in \cite{Jones3}.  

Let $K$ be any field and $f(x)\in K[x]$ a degree $d$ polynomial. We can give the set containing $0$ and the roots of $(f'\circ f^n)$ for $n\geq 0$ the structure of a tree graph rooted at $0$ as follows. The first level of the tree consists of the $d-1$ roots of $f'(x)$. Each root of $f'(f^n(x))=0$ lies in the $n+1$-st level of the tree. For each node $\gamma$ at level $n+1$ with $n\geq 1$ we draw an edge to the node $f(\gamma)$ and for each node $\beta$ at the 1-st level of the tree we draw an edge to $0=f'(\beta)$. Assume $f'\circ f^n$ is separable for every $n\geq 0$. This gives us a rooted tree with $d-1$ branches above the root point $0$ and $d$ branches above each of the other nodes in tree.

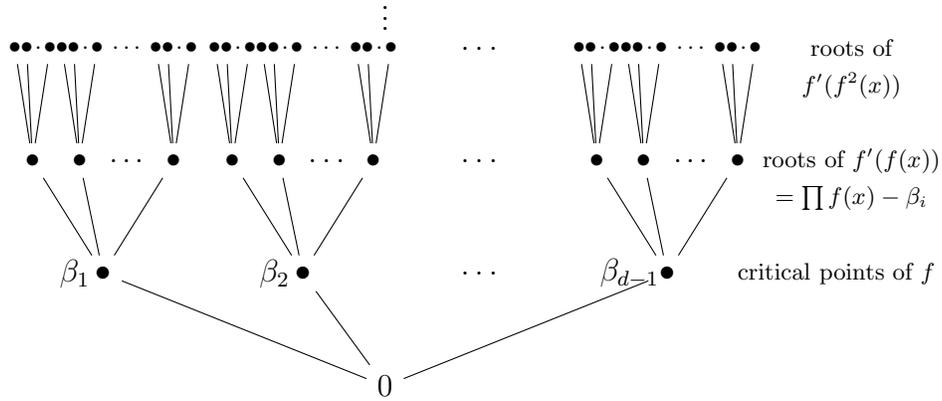
\begin{figure}[h]
\begin{tikzpicture}[scale = 1]
		\path 
		(5,1)node(0){$0$}
		(10/8,2.5)node(1){$\bullet$}
		(30/8+20/128,2.5)node(2){$\bullet$}
		(50/8,2.5)node(m){$\ldots$}
		(70/8,2.5)node(3){$\bullet$}
		(11,2.5)node{\scriptsize{critical points of $f$}};
		
		\draw(1)to(0);
		\draw(2)to(0);
		\draw(3)to(0);
		
		\path	
		(7/8,2.5)node(a_1){$\beta_1$}
		(27/8+20/128,2.5)node(a_2){$\beta_2$}
		(66/8,2.5)node(a_d){$\beta_{d-1}$}
		(10/32,4)node(11){\small$\bullet$}
		(30/32,4)node(12){\small$\bullet$}
		(50/32,4)node(1m){\small$\ldots$}
		(70/32,4)node(13){\small$\bullet$}
		
		(90/32+20/128,4)node(21){\small$\bullet$}
		(110/32+20/128,4)node(22){\small$\bullet$}
		(130/32+20/128,4)node(2m){$\ldots$}
		(150/32+20/128,4)node(23){\small$\bullet$}
		
		(200/32,4)node(mm){$\ldots$}
		
		(250/32,4)node(31){\small$\bullet$}
		(270/32,4)node(32){\small$\bullet$}
		(290/32,4)node(3m){\small$\ldots$}
		(310/32,4)node(33){\small$\bullet$}
		(11.2,4)node(text1){\scriptsize{roots of $f'(f(x))$}}
		
		(11.2,3.5)node(text1'){\scriptsize{$=\prod f(x)-\beta_i$}};
		
		\draw(11)to(1)to(12);
		\draw(13)to(1);
		\draw(21)to(2)to(22);
		\draw(23)to(2);
		\draw(31)to(3)to(32);
		\draw(33)to(3);
		
		\path(10/128,5.5)node(111){\scriptsize$\bullet$}
		(30/128,5.5)node(112){\scriptsize$\bullet$}
		(50/128,5.5)node(11m){\scriptsize$\ldots$}
		(70/128,5.5)node(113){\scriptsize$\bullet$}
		
		(90/128,5.5)node(121){\scriptsize$\bullet$}
		(110/128,5.5)node(122){\scriptsize$\bullet$}
		(130/128,5.5)node(12m){\scriptsize$\ldots$}
		(150/128,5.5)node(123){\scriptsize$\bullet$}
		
		(200/128,5.5)node(1mm){\scriptsize$\ldots$}
		
		(250/128,5.5)node(131){\scriptsize$\bullet$}
		(270/128,5.5)node(132){\scriptsize$\bullet$}
		(290/128,5.5)node(13m){\scriptsize$\ldots$}
		(310/128,5.5)node(133){\scriptsize$\bullet$}

		(330/128+20/128,5.5)node(211){\scriptsize$\bullet$}
		(350/128+20/128,5.5)node(212){\scriptsize$\bullet$}
		(370/128+20/128,5.5)node(21m){\scriptsize$\ldots$}
		(390/128+20/128,5.5)node(213){\scriptsize$\bullet$}
		
		(410/128+20/128,5.5)node(221){\scriptsize$\bullet$}
		(430/128+20/128,5.5)node(222){\scriptsize$\bullet$}
		(450/128+20/128,5.5)node(22m){\scriptsize$\ldots$}
		(470/128+20/128,5.5)node(223){\scriptsize$\bullet$}
		
		(520/128+20/128,5.5)node(2mm){\scriptsize$\ldots$}
		
		(570/128+20/128,5.5)node(231){\scriptsize$\bullet$}
		(590/128+20/128,5.5)node(232){\scriptsize$\bullet$}
		(610/128+20/128,5.5)node(23m){\scriptsize$\ldots$}
		(630/128+20/128,5.5)node(233){\scriptsize$\bullet$}
		
		(5,6)node{\vdots}
		
		(200/32,5.5)node(mmm){$\ldots$}

		(970/128,5.5)node(311){\scriptsize$\bullet$}
		(990/128,5.5)node(312){\scriptsize$\bullet$}
		(1010/128,5.5)node(31m){\scriptsize$\ldots$}
		(1030/128,5.5)node(313){\scriptsize$\bullet$}
		
		(1050/128,5.5)node(321){\scriptsize$\bullet$}
		(1070/128,5.5)node(322){\scriptsize$\bullet$}
		(1090/128,5.5)node(32m){\scriptsize$\ldots$}
		(1110/128,5.5)node(323){\scriptsize$\bullet$}
		
		(1160/128,5.5)node(3mm){\scriptsize$\ldots$}
		
		(1210/128,5.5)node(331){\scriptsize$\bullet$}
		(1230/128,5.5)node(332){\scriptsize$\bullet$}
		(1250/128,5.5)node(33m){\scriptsize$\ldots$}
		(1270/128,5.5)node(333){\scriptsize$\bullet$}
		(11.2,5.5)node(text){\scriptsize{roots of}}
		(11.2,5)node(text){\scriptsize{$f'(f^2(x))$}};
		
		\draw(111)to(11);
		\draw(112)to(11)to(113);
		\draw(121)to(12);
		\draw(122)to(12)to(123);
		\draw(131)to(13);
		\draw(132)to(13)to(133);
		\draw(211)to(21);
		\draw(212)to(21)to(213);
		\draw(221)to(22);
		\draw(222)to(22)to(223);
		\draw(231)to(23);
		\draw(232)to(23)to(233);
		\draw(311)to(31);
		\draw(313)to(31);
		\draw(312)to(31);
		\draw(321)to(32);
		\draw(322)to(32)to(323);
		\draw(331)to(33);
		\draw(332)to(33)to(333);
	\end{tikzpicture} 
	\caption{Pre-image tree for $\bigcup_{n} (f'\circ f^n)^{-1}(0))$}
\end{figure}

Since $f'\circ f^n$ is separable for every $n$, the extensions $K \left(\left(f'\circ f^n\right)^{-1}(0)\right)$ are Galois extensions of $K$. Since any element of $\Gal\left(K \left(\left(f'\circ f^n\right)^{-1}(0)\right)/K \right)$ commutes with $f$ and $f'$, $\Gal\left(K \left(\left(f'\circ f^n\right)^{-1}(0)\right)/K \right)$ acts on the tree up to the $n+1$-st level by permuting the branches and hence is isomorphic to a subgroup of $\Aut(T_{n+1})\cong S_{d-1}[[S_d]^n]$, where $\Aut(T_{n+1})$ denotes the automorphism group of the tree truncated at the $n+1$-st level and $[S_d]^n$ denotes the $n$-th iterated wreath product of $S_d$ with itself.

\section{Preliminary Results}\label{section:preliminaries}

The results in this section will be used in the proof of Theorem \ref{thm:main}.

\begin{lemma}\label{lemma:transposition} Let $A$ be a Dedekind domain with field of fractions $K$, $F(x)\in A[x]$ and $M=K(F^{-1}(0))$. Suppose $\fp||\Delta(F(x))$ in $A$, where $\Delta(F(x))$ denotes the polynomial discriminant of $F(x)$. Then for any prime $\fq$ of $M$ lying over $\fp$, the action of the inertia group $I(\fq|\fp)$ on the roots of $F(x)$ consists of a single transposition. 
\end{lemma}

\begin{proof}
Let $\theta$ be a root of $F(x)$, $L=K(\theta)$, and $B$ the integral closure of $A$ in $L$. 
By \cite[Chapter 3, Proposition 13]{Lang}, \[\Delta(F(x))=c^2\Delta(B/A)\] for some $c\in A$ . Hence since $\fp||\Delta(F(x))$, we have $\fp||\Delta(B/A)$. 

Suppose $\fp B=\prod \fP_i^{e_i}$. Let $f_i=f(\fP_i|\fp)$, the residue degree of $\fP_i$ over $\fp$, and $e_i=e_i(\fP_i|\fp)$, the inertia degree of $\fP_i$ over $\fp$. Then the power of $\fp$ in $\Delta(B/A)$ is greater than or equal to $\sum (e_i-1)f_i$ with equality if $\fp$ is tamely ramified, see \cite[Chapter 3, Proposition 8]{Lang} and \cite[Chapter 3, Proposition 14]{Lang}. Hence we must have $e_i=2$ for exactly one value of $i$ and $f_i=1$ for that $i$.

Now let $D(\fq|\fp)$ and $I(\fq|\fp)$ denote the decomposition group and inertia group of $\fq$ over $\fp$ respectively. By \cite[Lemma 2.4]{JO} and \cite[Remark 2.5]{JO} 
there is a bijection between the orbits of the roots of $F(x)$ under the action of $D(\fp|\fq)$  and the set of extensions $\fP$ of $\fp$ to $L$ with the property: if $\fP$ corresponds to the orbit $Y$, then the length of $Y$ is $e(\fP|\fp)f(\fP|\fp)$ and $Y$ decomposes into the union of $f(\fP|\fp)$ orbits of length $e(\fP|\fp)$ under the action of $I(\fq|\fp)$. Therefore by the previous paragraph, $I(\fq|\fp)$ acts on the roots of $F(x)$ as a single transposition. 
\end{proof}

\begin{lemma}[\cite{AHM}, Proposition 3.2]\label{lemma:discriminant} If $f(x)$ is a monic polynomial of degree $d$, then \[\Delta(f^{n+1}(x)-\alpha)=\pm\Delta(f^{n}(x)-\alpha)^d\prod_{f'(b)=0}(f^{n+1}(b)-\alpha)^{m_b}\] where $m_b$ is the multiplicity of $c$ as a root of $f'(x)$.
\end{lemma}

\begin{lemma}[\cite{odoni}, Lemma 2.4]\label{lemma:specialization_subgroup}
Let $I$ be an integrally closed domain with field of fractions $F$ and suppose that $f(x)=a_0+\dots+a_kx^k\in I[x]$ where $k\geq 1$ and $a_k\neq 0$. Let $\varphi$ be any morphism mapping $I$ to $K$ where $K$ is a field with $a_k\notin\ker\varphi$. Let $\tilde{\varphi}$ be the induced map $I[x]\rightarrow K[x]$. If $\tilde{\varphi}(f(x))$ is separable over $K$ then
\begin{enumerate}[(i)]
\item $f(x)$ is separable over $F$,
\item $\Gal(K((\tilde{\varphi}(f))^{-1}(0))/K)$ is isomorphic to a subgroup of $\Gal(F(f^{-1}(0))/F)$.
\end{enumerate}
\end{lemma}

The following result is similar to \cite[Lemma 2.3]{BeJu}. It is modified here to fit our situation. Here and throughout the paper we will use the shorthand of referring to a prime of a number field to mean a prime of the ring of integers of the number field.

\begin{lemma}\label{lemma:m_transitive_subgroup} Let $K$ be a number field, $d\geq 4$, and $m=d-1$ or $d-2$ with $m\geq 3$. Let $b,x_0\in K$, and suppose there is a prime $\fp$ of $K$ such that $\fp\nmid (d-m)$, $v_\fp(b)\in \{-1,-2\}$, $v_\fp(x_0)\geq 1$, $(d-m)\mid v_\fp(b)$, $v_\fp(x_0/b)<m$ and $\gcd(m,v_\fp(x_0/b))=1$. Let $f(x)=x^d-bx^m-x_0$, let $n\geq 0$, and let $\gamma\in f^{-n}(0)$. Suppose that $f^n(x)\in K[x]$ is irreducible over $K$. Then there is a prime $\fP$ of $K(\gamma)$ lying above $\fp$, and a prime $\fQ$ of $K(f^{-1}(\gamma))$ such that 
\begin{itemize}
\item $\fP$ has ramification index $m^n$ over $\fp$,
\item $m^nv_\fP((\gamma+x_0)/b)$ is a positive integer relatively prime to $m$ and $<m^{n+1}$, where $v_\fP$ is the $\fP$-adic valuation on $K(\gamma)$ extending $v_\fp$,
\item $\fQ$ lies above $\fP$, and
\item the ramification group $I(\fQ|\fP)$ acts transitively on $m$ roots of $f(x)-\gamma$ and fixes the other $d-m$ roots.
\end{itemize}
\end{lemma}

\begin{proof} We first prove the first two bullet points by induction on $n$. If $n=0$, then $\gamma=0$ and both statements follow trivially taking $\fP=\fp$.

Now suppose the first two statements hold for $n-1$. Let $\gamma_{n-1}=f(\gamma)\in f^{-(n-1)}(0)$. Since $f^n(x)$ is irreducible over $K$, it follows from Capelli's lemma that $f(x)-\gamma_{n-1}$ is irreducible over $K(\gamma_{n-1})$. Also, by our induction hypothesis we have some prime $\fP_{n-1}$ of $K(\gamma_{n-1})$ with ramification index $m^{n-1}$ over $\fp$ and $m^{n-1}v_{\fP_{n-1}}((\gamma_{n-1}+x_0)/b)$ is a positive integer relatively prime to $m$ and less than $m^{n}$.

The Newton polygon of $f(x)-\gamma_{n-1}=x^d-bx^m-(x_0+\gamma_{n-1})$ at the prime $\fP_{n-1}$ has a line segment of length $m$ with slope $-m^{-1}v_{\fP_{n-1}}((\gamma_{n-1}+x_0)/b)$ and a line segment of slope $-v_{\fP_{n-1}}(b)/(d-m)=-v_\fp(b)/(d-m)$. This implies $f(x)-\gamma_{n-1}$ factors as a product of two polynomials say $h_1$ and $h_2$ over $K(\gamma_{n-1})_{\fP_{n-1}}$ with $\deg h_1=m$ and $\deg h_2=d-m$. Also, the Newton polygon for $h_1$ is a single line segment of slope $-m^{-1}v_{\fP_{n-1}}((\gamma+x_0)/b)$. Applying the induction hypothesis, we have $-m^{-1}v_{\fP_{n-1}}((\gamma_{n-1}+x_0)/b)=N/m^n$ for a positive integer $N$ relatively prime to $m$. Therefore $K(\gamma_{n-1})$ has a prime $\fP$ of ramification index $m$ over $\fP_{n-1}$. By the induction hypothesis, the ramification index of $\fP$ over $\fp$ is $m^n$, finishing the proof of the first bullet point.

Let $v_\fP$ be the $\fP$-adic valuation on $K(\gamma)$ extending $v_{\fP_{n-1}}$. Applying an element $\sigma\in \Gal(K(f^{-1}(\gamma_{n-1})/K(\gamma_{n-1}))$ if necessary, we may assume $\gamma$ is a root of $h_1$. Then $v_\fP(\gamma)=N/m^n$. Further, since $v_\fP(x_0)=v_\fp(x_0)\geq 1>N/m^n$ and $v_\fP(b)=v_\fP(b)$ is a negative integer, we have $v_{\fP}(\gamma+x_0)=v_\fP(\gamma)$ and $m^nv_\fP((\gamma+x_0)/b)=m^n(N/m^n)-m^nv_\fP(b)=N-m^nv_\fP(b)$ is an integer relatively prime to $m$. Since $N<m^n$, $-v_\fP(b)=1$ or $2$, and $m\geq 3$, we see $N-m^nv_\fP(b)<m^{n+1}$. This finishes the proof of the second bullet point.

We now prove the third and fourth bullet points using the first and second. Let $\fP$ be a prime of $K(\gamma)$ satisfying the first two bullet points. 
By the argument above, $f(x)-\gamma$ factors a as a product of two polynomials $h_1(x)h_2(x)$ over $K(\gamma)_\fP$ where the Newton polygon of $h_1$ consists of a single line segment of length $m$ and slope $m^{-1}v_\fp((\gamma+x_0)/b)$. Similarly, the Newton polygon of $h_2$ consists of a single line segment of length $d-m$ and slope $\ell:=-v_\fP(b)/(d-m)$. 

Let $\pi$ be a uniformizer for $\fp$ in $\cO_K$ and consider the polynomial
\[\pi^{d\ell}(f(\pi^{-\ell}x)-\gamma)=x^d-\pi^{(d-m)\ell}bx^m-\pi^{d\ell}(x_0+\gamma)\]
which is defined over $K(\gamma)$ and, since $v_\fp(b)=-\ell(d-m)$, has $K(\gamma)_\fP$-integral coefficients. 
Then the monic polynomials $\pi^{m\ell}h_1(\pi^{-\ell}x)$ and $\pi^{(d-m)\ell} h_2(\pi^{-\ell x})$ have  $K(\gamma)_\fP$-integral coefficients as well. Reducing modulo $\fP$ we have \[\pi^{d\ell}(f(\pi^{-\ell}x)-\gamma)\equiv x^m(x^{d-m}-c)\mod \fP\] where $c\not\equiv 0\mod \fP$. Therefore,
\[\pi^{(d-m)\ell}h_2(\pi^{-\ell}x)\equiv x^{d-m}-c\mod \fP.\] Then since $p\nmid (d-m)$ the splitting field of $h_2$ is unramified over $\fP$. 

Let $L:=K(f^{-1}(\gamma))$ and let $\fQ$ be any prime of $L$ lying over $\fP$. Let $\delta_1,\dots,\delta_d$ denote the roots of $f(x)-\gamma$ in $L_\fQ$. Without loss of generality, we may assume $\delta_1,\dots,\delta_m$ are the roots of $h_1(x)$ and $\delta_{m+1},\dots,\delta_d$ are the roots of $h_2(x)$. Then $\delta_{m+1},\dots, \delta_d$ lie in an intermediate unramified extension $L'_\fQ$ of $K(\gamma)_\fP$. The roots $\delta_1,\dots,\delta_m$ correspond to the line segment of length $m$ of the Newton polygon and therefore $g$ is irreducible over $L'_\fQ$ and the extension $L_\fQ/L'_\fQ$ is totally ramified. 

Then since $\Gal(L_\fQ/K(\gamma)_\fP)$ is isomorphic to the decomposition group $D(\fQ|\fP)$ and $\Gal(L_\fQ/L'_\fQ)$ is isomorphic to the interia group $I(\fQ|\fP)$, it follows that $I(\fQ|\fP)$ acts transitively on $m$ roots of $f(x)-\gamma$ and fixes the remaining $d-m$.
\end{proof}

\begin{lemma}[\cite{BeJu}, Lemma 2.4]\label{lemma:getting_Sd} Let $d\geq 3$, let $m$ be an integer relatively prime to $d$ with $d/2<m<d$ and let $G\subseteq S_d$ be a subgroup that contains a transposition, acts transitively on $\{1,2,\dots, d\}$, and has a subgroup that acts trivially on $\{m+1,\dots, d\}$ and transitively on $\{1,2,\dots,m\}$. Then $G=S_d$.
\end{lemma}

\section{Proof of Theorem \ref{thm:main}}\label{section:proof}

Consider $f'(x)=dx^d+(d-1)s_1x^{d-1}+\dots+s_{d-1}$. Note, except for the leading coefficient, the coefficients of $f'(x)$ are independent transcendentals over $K$ (and $K_{t}$), hence $\frac{1}{d}f'(x)$ is a generic monic polynomial. Let $\beta_1, \dots, \beta_{d-1}$ denote the roots of $f'(x)$, that is  $f'(x)=d\prod(x-\beta_i)$. Define $K_\beta:=K(\beta_1,\dots, \beta_{d-1})$ and $K_{\beta,t}:=K_\beta(t)$, the splitting field of $f'(x)$ over $K_s$ and $K_{s,t}$ respectively. Then by a classical result
\[\Gal\left(K_{\beta}/K_{s}\right)\cong\Gal\left(K_{\beta,t}/K_{s,t}\right)\cong S_{d-1}.\]
Thus Theorem \ref{thm:main} follows from Proposition \ref{proposition:main}.

\begin{lemma}\label{lemma:irreducible} If $i\neq j$, then $f^n(\beta_i)-\beta_j\in K(\beta_1,\beta_2,\dots \beta_{i-1},\beta_{i+1},\dots,\beta_{d-1},t)[\beta_i]$ is irreducible.
\end{lemma}

\begin{proof}  Note, if $f^n(\beta_{i})-\beta_j$ has a nontrivial factorization, then for any specialization of the coefficients, the image of $f^n(\beta_{i})-\beta_j$ has a nontrivial factorization as well. Specialize $\beta_\ell\mapsto 0$ for $\ell\neq i$ and $t$ to some prime $\pi$ of the ring of integers $\cO_K$ of $K$ dividing $d$. The image of $f^n(\beta_{i})-\beta_j$ is $\bar{f}^n(\beta_{i})\in K[\beta_{i}]$ where $\bar{f}(x)=x^d-\frac{d}{d-1}\beta_{i}x^{d-1}+\pi$. 

We claim $\bar{f}^n(\beta_{i})$ is Eisenstein in the variable $\beta_i$ at $\pi$. We prove this by induction. First note, $\bar{f}(\beta_{i})=-\frac{1}{d-1}\beta_{i}^d+\pi$. Now suppose $\bar{f}^n(\beta_{i}) = c_{d^n}\beta_{i}^{d^n}+\dots+ c_1\beta_{i}+c_0$ where $v_\pi(c_{d^n})=0$, $v_\pi(c_k)\geq 1$ for $1\leq k\leq d^n-1$, and $v_\pi(c_0)=1$. Then $\bar{f}^{n+1}(\beta_{i})=(c_{d^n}\beta_{i}^{d^n}+\dots +c_1\beta_{i}+c_0)^d-\frac{d}{d-1}\beta_{i}(c_{d^n}\beta_{i}^{d^n}+\dots +c_1\beta_{i}b+c_0)^{d-1}+\pi = c_{d^n}^d\beta_{i}^{d^{n+1}}+\pi\beta_{i}p(\beta_i)+(c_0^d+\pi)$ is Eisenstein at $\pi$ as well.
\end{proof}

We fix the following notation which will be used throughout the remainder of this section. Let $\gamma\in (f'\circ f^n)^{-1}(0)$. Define $M_\gamma = K_{\beta,t}(f^{-1}(\gamma))$ and $\hat{M_\gamma}=\prod_{\delta\in (f'\circ f^n)^{-1}(0)\setminus\{\gamma\}} M_\delta$. 

\begin{proposition}\label{prop:transposition} Fix $n\geq 1$ and let  $\gamma\in (f'\circ f^{n-1})^{-1}(0)$.  There is an automorphism of $K_{\beta,t}(\cup_j f^{-{n}}(\beta_j))$ that fixes $\hat{M_\gamma}$ and acts as a single transposition on the roots of $f(x)-\gamma$.
\end{proposition}

\begin{proof} Say $f^{n-1}(\gamma)=\beta_j$ and fix some $i\neq j$. Lemma \ref{lemma:irreducible} shows $\fp:=(f^n(\beta_{i})-\beta_j)$ is prime in $K(\beta_1,\beta_2,\dots \beta_{i-1},\beta_{i+1},\dots,\beta_{d-1},t)[\beta_i]$. Now we note $\deg_{\beta_{i}}(f^n(\beta_{i})-\beta_k)=d^n$, $\deg_{\beta_{i}}(f^n(\beta_\ell)-\beta_k)<d^n$ for $\ell\neq i$ and all $k$, and  $\deg_{\beta_{i}}(f^m(\beta_\ell)-\beta_k)<d^n$, for all $\ell$ and $m<n$. Thus, $f^n(\beta_{i})-\beta_j$ does not divide $f^m(\beta_{\ell})-\beta_k$ for any $m\leq n$ and any $\ell,k$ unless $m=n$, $\ell=i$, and $k=j$. 

Now by Lemma \ref{lemma:discriminant}, we see $\fp\nmid \Delta(f^n(x)-\beta_k)$ for $k\neq j$ and $\fp||\Delta(f^n(x)-\beta_j)$. So $\fp$ does not ramify in  $K_{\beta,t}(f^{-{n}}(\beta_k))$ for $k\neq j$. Lemma \ref{lemma:transposition} implies for any prime $\fq$ of the integral closure of $K(\beta_1,\beta_2,\dots \beta_{i-1},\beta_{i+1},\dots,\beta_{d-1},t)[\beta_i]$ in $K_{\beta,t}(f^{-{n}}(\beta_j))$ lying over $\fp$, $I(\fq|\fp)$ consists of a single transposition of the roots of $(f^{-n}(x)-\beta_j)$. 

Now let $\fP$ be any prime of $M_\gamma$ dividing $f(\beta_i)-\gamma$ and note $\fP$ lies over $\fp$. Also, $(f(\beta_i)-\gamma)|\Delta(f(x)-\gamma)$, so $\fP$ is ramified over $\fp$. Let $\fq$ be a prime of $K_{\beta,t}(f^{-{n}}(\beta_j))$ lying over $\fP$. By the arguments in the previous paragraph $|I(\fq|\fp)|=2$. But we also have $|I(\fq|\fP)|\cdot|I(\fP|\fp)|=|I(\fq|\fp)|=2$. Hence $|I(\fP|\fp)|=2$ and since the elements of  $I(\fP|\fp)$ extend to elements of $I(\fq|\fp)$, the nontrivial element of $I(\fP|\fp)$ and hence $I(\fq|\fp)$ must act as a single transposition of the roots of $f(x)-\gamma$. 

Finally, let $\fQ$ denote any extension of $\fq$ to $K_{\beta,t}(\cup_j f^{-{n}}(\beta_j))$. Note since $\fp$ does not ramify in $K_{\beta,t}(f^{-{n}}(\beta_k))$ for $k\neq i$ and $I(\fQ|\fp)$ acts trivially on all of $\hat{M_\gamma}$.
\end{proof}

\begin{proposition}\label{prop:one_level}  Fix $n\geq 1$ and let $\gamma\in (f'\circ f^{n-1})^{-1}(0)$.  Then \[\Gal(M_\gamma/K_{\beta,t}(\gamma))\cong S_d.\]
\end{proposition}

\begin{proof} 

Let $\beta = f^{n-1}(\gamma)$. We first note $f^{n-1}(x)-\beta$ is irreducible over $K_{\beta,t}$ since it is easy to find a specialization that is irreducible, take for example any specialization to $\cO_K$ that sends $f(x)$ to an Eisenstein polynomial with linear term $0$ and $\beta$ to $0$. Then Capelli's Lemma implies $f(x)-\gamma$ is irreducible over $K_{\beta,t}(\gamma)$. Therefore, $\Gal(M_\gamma/K_{\beta,t}(\gamma))$ isomorphic to a transitive subgroup of $S_d$. Also, by restricting the action of the transposition from Proposition \ref{prop:transposition} to $M_\gamma$, we see $\Gal(M_\gamma/K_{\beta,t}(\gamma))$ contains a transposition. This is enough to conclude the result if $d=2$, $d=3$, or $d=5$ (or more generally if $d$ is prime).

When $d=4$ or $d\geq 6$, we show there is a specialization of $f(x)$ satisfying the hypotheses of Lemma \ref{lemma:m_transitive_subgroup} with $m$ satisfying the hypotheses of Lemma \ref{lemma:getting_Sd}. The result then follows from Lemma \ref{lemma:specialization_subgroup}.

We choose a specialization of $\beta_1,\dots,\beta_{d-1},t$ to $K$ as follows. If $d$ is even, choose $m=d-1$, $\fp$ any prime of $\cO_K$, $b=\fq/\fp$ where $\fq$ is any other prime of $\cO_K$, and $x_0=\fp\fq$. If $d$ is odd, choose $m=d-2$, $\fp$ any prime of $\cO_K$ not dividing $2$, $b=\fq/\fp^2$ where $\fq$ is any other prime of $\cO_K$, and $x_0=\fq\fp^2$. Specialize $\beta=f^{n-1}(\gamma)$ to $0$, $t$ to $-x_0$ and $f(x)$ to $x^d-bx^m-x_0$. In either case, $f(x)$ is Eisenstein at $\fq$ and hence so is $f^m(x)$ for all $m$ and so $f^{n-1}(x)$ is irreducible over $K$. Further, in either case the other hypotheses of Lemma \ref{lemma:m_transitive_subgroup} and \ref{lemma:getting_Sd} are satisfied, as desired.
\end{proof}

\begin{proposition}\label{prop:extending_basefield}  Fix $n\geq 1$ and let $\gamma\in (f'\circ f^{n-1})^{-1}(0)$. Then \[\Gal(K_{\beta,t}(\cup_j f^{-n}(\beta_j))/\hat{M_\gamma})\cong S_d.\]
\end{proposition}

\begin{proof} We now show we can extend our result from Proposition \ref{prop:one_level} to the base field  $\hat{M_\gamma}$. Note, $\Gal(K_{\beta,t}(\cup_j f^{-{n}}(\beta_j))/\hat{M_\gamma})\cong \Gal(M_\gamma/M_\gamma\cap \hat{M_\gamma})$ so it suffices to show $\Gal(M_\gamma/M_\gamma\cap \hat{M_\gamma})\cong S_d$. Since $\hat{M_\gamma}$ is a Galois extension of $K_{\beta,t}(\gamma)$,   $M_\gamma\cap \hat{M_\gamma}$ is also a Galois extension of $K_{\beta,t}(\gamma)$, so $\Gal(M_\gamma/M_\gamma\cap \hat{M_\gamma})$ is a normal subgroup of $S_d$. It also contains a transposition, namely the restriction of the automorphism from Proposition \ref{prop:transposition} to $M_\gamma$. Therefore $\Gal(M_\gamma/M_\gamma\cap \hat{M_\gamma})\cong S_d$ as desired.
\end{proof}

We are now ready to complete the proof of Theorem \ref{thm:main} by proving Proposition \ref{proposition:main}.

\begin{proof}[Proof of Proposition \ref{proposition:main}]
Let $K_n=K_{s,t}\left(\left(f'\circ f^n\right)^{-1}(0)\right)=K_{\beta,t}(\cup_j f^{-{n}}(\beta_j))$, for $n\geq 0$. Note, with this notation $K_0=K_{\beta ,t}$.
Proposition \ref{prop:extending_basefield} implies for all $n\geq 1$,
 \[\left|\Gal(K_n/K_{n-1})\right|=\left|\prod_{\gamma\in \cup_j f^{-(n-1)}(\beta_j)} S_d\right|=d!^{d^{n-1}(d-1)}.\]
As we observed at the beginning of the section, we have \[\Gal\left(K_{\beta,t}/K_{s,t}\right)\cong S_{d-1}.\]
Thus by induction on $n$, we have 
\begin{align*}|\Gal\left(K_n/K_{s,t}\right)|&=|\Gal(K_n/K_{n-1})|\cdot |\Gal\left(K_{n-1}/K_{s,t}\right)|\\
&=d!^{d^{n-1}(d-1)}\cdot |S_{d-1}[[S_d]^{n-1}]|\\
&=|S_{d-1}[[S_d]^n]|.
\end{align*} 
Since we know $\Gal\left(K_{s,t}\left(\left(f'\circ f^n\right)^{-1}(0)\right)/K_{s,t}\right)$ is isomorphic to a subgroup of $S_{d-1}[[S_d]^n]$, we have $\Gal\left(K_{s,t}\left(\left(f'\circ f^n\right)^{-1}(0)\right)/K_{s,t}\right)\cong S_{d-1}[[S_d]^n]$.
\end{proof}

\begin{corollary}\label{corollary:specialization1} Fix integers $d,n$ with $d>1$ and $n\geq 0$. There is a Zariski dense subset $\cH$ of $K^d$ such that for all $f(x)=x^d+c_{d-1}x^{d-1}+\dots+c_0$ with $(c_0,\dots,c_{d-1})\in \cH$
\[\Gal\left(K\left(\left(f'\circ f^n\right)^{-1}(0)\right)/K \right)\cong S_{d-1}[[S_d]^n].\]
\end{corollary}

\begin{corollary}\label{corollary:specialization2} Fix integers $d,n$ with $d>1$ and $n\geq 0$. There is a Zariski dense subset $\cH$ of $K^d$ such that for all degree $d$ monic polynomials $f(x)$ with critical points $b_1,\dots, b_{d-1}$ and constant term $c_0$ with $(b_1,\dots,b_{d-1},c_0)\in \cH$
\[\Gal\left(K\left(\left(f'\circ f^n\right)^{-1}(0)\right)/K \right)\cong \prod_{d-1}[S_d]^n.\]
\end{corollary}

\begin{proof}[Proof of Corollaries \ref{corollary:specialization1}, \ref{corollary:specialization2}] The statements follow from Hilbert's irreducibility theorem for number fields applied to the Galois resolvent of $f'\circ f^n(x)$ in $K_{s,t}[x]$ and $\prod_i (f^n(x)-\beta_i)$ in $K_{\beta,t}[x]$ respectively (see for example \cite[Lemma 6.1]{odoni}). Note if the critical points of $f(x)$ are defined over $K$ as in Corollary \ref{corollary:specialization2}, the polynomial $f'\circ f^n(x)$ factors as $\prod_i (f^n(x)-b_i)$ over $K$ and is a specialization of $\prod_i(f^n(x)-\beta_i)\in K_{\beta,t}[x]$. 
\end{proof}

\section{Attracting Periodic Points and the Dynamical Mordell-Lang Conjecture}\label{section:application}

We now apply the results of the previous section to get a bound on the proportion of primes for which a polynomial $f(x)\in K[x]$ with critical points in $K$ has a $\fp$-adic attracting periodic point. We say a periodic point $P\in K_\fp$ of exact period $n$ is \textit{attracting} if $|(f^n)'(P)|_\fp<1$. Note, this is equivalent to the existence of a neighborhood $U$ of $P$ such that $\lim_{m\rightarrow\infty}|f^{nm}(Q)-P|_\fp=0$ for all $Q\in U$ (see \cite{R-L} or \cite{BenedettoBook}). For a group $G$ acting on a set $X$ we define the \textit{fixed point proportion} of $G$, denoted $\FPP(G)$ to be $|\{g\in G: g(x)=x \text{ for some } x\in X\}|$.

\begin{theorem}\label{thm:boundonattracting}
Let $K$ be a number field and $f(x)\in K[x]$ such that all finite critical points of $f(x)$ are defined over $K$. Then for any positive integer $m$, the density of the set of primes $\fp$ for which $f$ has a finite attracting periodic point in $\bar{K_\fp}$ is bounded above by $\FPP\left(\Gal(K((f'\circ f^m)^{-1}(0))/K)\right)$, the proportion of elements of the Galois group whose action on the roots of $f'\circ f^m(x)$ has a fixed point.
\end{theorem}

\begin{proof} Let $S$ be the set of primes of good reduction for $f(x)$ such that $f$ has a finite attracting periodic point in $\bar{K}_\fp$. Let $\delta(S)$ denote the natural density of $S$. Since there are only finitely many primes of bad reduction it suffices to show $\delta(S)$ is bounded above by $\FPP(\Gal(K((f'\circ f^m)^{-1}(0))/K))$. 

We first show a prime if a prime $\fp$ of $K$ is in $S$, then the reduction of $f(x) \mod \fp$ has a periodic critical point. Suppose $\bar{K_\fp}$ is a periodic point of $f$ with exact period $n$. Let $K'$ be a finite extension of $K_\fp$ containing $P$ and let $\fP$ be the prime of $K'$ extending $\fp$. Then $|(f^n)'(P)|_\fP<1$, or equivalently  $(f^n)'(P)\equiv 0 \mod \fP$. Then using the chain rule we have, \[(f^n)'(P)= f'(P)f'(f(P))\dots f'(f^{n-1}(P))\equiv 0\mod \fP.\] This implies that the orbit of $P\mod \fP$ contains a critical point of $f(x)\mod \fP$. Since the critical points of $f(x)$ are all defined over $K$, each critical point of $f(x)\mod\fP$ is the reduction of a critical point of $f(x)$ and hence there is critical point $\beta\in K$ of $f(x)$ such that $\beta \mod \fp$ is a periodic point of $f(x) \mod \fp$.

Suppose $\fp\in S$, then some critical point $\beta$ of $f(x)$ is periodic modulo $\fp$, so there is some $n$ such that $f^n(\beta)\equiv\beta\mod \fp$. For any positive integer $m$, let $k$ be an integer such that $nk\geq m$. Then, $f^m(f^{nk-m}(\beta))\equiv f^{nk}(\beta)\equiv \beta\mod \fp$. Hence $f^{nk-m}(\beta)\in K$ is a solution to the congruence$(f'\circ f^m)(x)=\prod_{f'(\beta)=0} f^m(x)-\beta\equiv 0 \mod \fp$. 
Thus, the density of $S$ is bounded above by the density of primes of good reduction for which $f'(f^{m}(x))\equiv 0 \mod \fp$ has a solution for any $m$. 

With $m$ fixed, we split $S$ into two sets $S_1=\{\fp\in S: \fp|\Delta(f'\circ f^m(x))\}$ which contains all primes of $S$ which ramify in $K((f'\circ f^m)^{-1}(0))$ and $S_2=S\setminus S_1$. The set $S_1$ is finite, hence it suffices to show $S_2$ has density bounded above by $\FPP(\Gal(K((f'\circ f^m)^{-1}(0))/K))$.  If $\fp\in S_2$, then it does not ramify in $K_m:=K((f'\circ f^m)^{-1}(0))$.  Let $\Frob\left(\frac{K_m/K}{\fp}\right)$ denote the Frobenius automorphism. Since $\fp\in S$, $f'(f^{m}(x))\equiv 0 \mod \fp$ has a linear factor. This implies  $\Frob\left(\frac{K_m/K}{\fp}\right)$ fixes some root of $f'\circ f^m(x)$. Hence Frobenius automorphisms for $\fp\in S_2$ are contained in the set of elements of $\Gal(K((f'\circ f^m)^{-1}(0))/K)$ with at least one fixed point. Then by the Chebotarev density theorem (see \cite[Theorem 6.3.1]{FJ}), $\delta(S)\leq \FPP(\Gal(K((f'\circ f^m)^{-1}(0))/K))$ as desired.
\end{proof}

\begin{remark} For $f(x)$ with all critical points in $K$, the polynomial $f'\circ f^m(x)$ factors as $\prod_{f'(\beta)=0} f^m(x)-\beta$. So we have $\Gal(K((f'\circ f^m)^{-1}(0))/K)\subseteq \prod_{d-1}[S_d]^m$, where $d=\deg f$. 
\end{remark}

\begin{example}
As an initial example, consider the case where $f(x)$ has a superattracting point $P$ in $K$ of exact period $n$. Then $|(f^n)'(P)|=0$, which implies some critical point $\beta$ of $f(x)$ is periodic. Then $\beta$ is a root of $f'\circ f^m(x)$ for infinitely many $m$ and since $\beta$ is in the base field $K$, every element of $\Gal(K((f'\circ f^m)^{-1}(0))/K)$ has a fixed point. That is $\FPP(\Gal(K((f'\circ f^m)^{-1}(0))/K))=1$, which is what we expect since $P$ is a $\fp$-adic attracting point for all primes $\fp$.
\end{example}

\begin{example}
Consider $f(x)=x^3+5$ and $K=\bQ$. In this case $f(x)$ has one finite critical point $0$, which is not periodic. The map $f:\bQ_p\rightarrow \bQ_p$, has a finite attracting periodic point if and only if $0$ is periodic for the map $f:\bF_p\rightarrow\bF_p$. It was shown in \cite{HJM} that the proportion of primes of $\bZ$ for which this holds is $\frac{1}{2}=\FPP(\Gal(\bQ(f^{-n}(0))/\bQ)$. The set contains all the $2\mod 3$ primes where the cubing map is a bijection and a set of $1\mod 3$ primes of density $0$. On  the other hand if we extend the base field to $K=\bQ(\zeta_3)$, then $\FPP(\Gal(K(f^{-n}(0))/K))\rightarrow 0$ and $n\rightarrow \infty$ and hence the density of primes for which $K_\fp$ has a $\fp$-adic attracting point is $0$.
\end{example}

In each of these examples $\Gal(K((f'\circ f^n)^{-1}(0))/K)$ is a small subset of  $\prod_{d-1}[S_d]^n$. However, by Corollary 4.6, ``most'' polynomials $f(x)$ with critical points in $K$ have $\Gal(K((f'\circ f^n)^{-1}(0))/K)\cong \prod_{d-1}[S_d]^n$. For example, one can show $f(x)=x^2+1$ is such an example by \cite{odoni} and \cite{stoll}.

\begin{proof}[Proof of Theorem \ref{thm:zariskisetforonemap}]
First note 
\begin{align*}
1-\FPP\left(\prod_{d-1}[S_d]^n\right)&=\frac{|\{\sigma\in \prod_{d-1}[S_d]^n: \sigma\text{ does not have a fixed point}\}|}{|\prod_{d-1}[S_d]^n|}\\
&=\frac{|\{(\sigma_1,\dots,\sigma_{d-1})\in \prod_{d-1}[S_d]^n: \sigma_i \text{ does not have a fixed point for all } i\}|}{|\prod_{d-1}[S_d]^n|}\\
&=\prod_{i} \frac{|\{\sigma_i\in [S_d]^n: \sigma_i \text{ does not have a fixed point}\}|}{|[S_d]^n|}\\
&=\left(1-\FPP\left([S_d]^n\right)\right)^{d-1}.
\end{align*}
Thus, \[\FPP\left(\prod_{d-1}[S_d]^n\right)=1-\left(1-\FPP\left([S_d]^n\right)\right)^{d-1}.\]
By \cite[Lemma 4.3]{odoni}, $\FPP([S_d]^n)\rightarrow 0$ as $n\rightarrow \infty$ and hence 
$1-\left(1-\FPP\left([S_d]^n\right)\right)^{d-1}\rightarrow 0$ as $n\rightarrow \infty$ as well. We can make this more precise, using \cite[Proposition 4.5]{J2}, which says $\FPP([S_d]^n)\leq \frac{2}{n+2}$. Then, using the fact that $(n+2)^{d-1}\leq n^{d-1}+C_dn^{d-2}$ for some positive constant $C_d$, we have
\begin{align*} &\FPP\left(\prod_{d-1}[S_d]^n\right)=1-\left(1-\FPP\left([S_d]^n\right)\right)^{d-1}
\leq 1-\left(1-\frac{2}{n+2}\right)^{d-1}\\&= 1-\left(\frac{n}{n+2}\right)^{d-1} \leq 1-\frac{n^{d-1}}{n^{d-1}+C_dn^{d-2}}=\frac{C_dn^{d-2}}{n^{d-1}+C_dn^{d-2}}\leq \frac{C_dn^{d-2}}{n^{d-1}}=\frac{C_d}{n}.
\end{align*}

Choose $n_0$ so that $\FPP\left(\prod_{d-1}[S_d]^n\right)<\epsilon$ then by Corollary \ref{corollary:specialization2}, there is a Zariski dense subset $\cH$ of $K^d$ such that for all monic polynomials $f(x)$ with critical points $b_1,\dots, b_{d-1}$ and constant term $c_0$ with $(b_1,\dots,b_{d-1},c_0)\in \cH$
\[\Gal\left(K\left(\left(f'\circ f^{n_0}\right)^{-1}(0)\right)/K \right)\cong \prod_{d-1}[S_d]^n.\] The result then follows from Theorem \ref{thm:boundonattracting}.
\end{proof}

Finally, we turn to the connection of $p$-adic attracting periodic points to the Dynamical Mordell-Lang Conjecture. The proof of Theorem \ref{thm:mainapplication} follows quickly from the following theorem of Benedetto, Ghioca, Kurlberg, and Tucker.

\begin{theorem}[Theorem 3.4, \cite{BGKT}]\label{thm:avoidingramification} Let $V$ be a subvariety of $(\bP^1)^g$ defined over $\bC_p$, let $f_1,\dots, f_g\in \bC_p(t)$ be rational functions of good reduction on $\bP^1$, and let $\Phi$ denote the coordinatewise action of $(f_1,\dots, f_g)$ on $(\bP^1)^g$. Let $\cO$ be the $\Phi$-orbit of a point $\alpha=(x_1,\dots,x_g)\in (\bP^1(\bC_p))^g$, and suppose that for each $i$, the orbit $\cO_{f_i}(x_i)$ does not intersect the residue class of any attracting $f_i$-periodic point. Then $V(\bC_p)\cap \cO$ is a  union of at most finitely many orbits of the form $\{\Phi^{nk+\ell}(\alpha)\}_{n\geq 0}$ for nonnegative integers $k$ and $\ell$.
\end{theorem}

%

\begin{proof}[Proof of Theorem \ref{thm:mainapplication}]
Let $\epsilon =\frac{1}{g}$. By Theorem \ref{thm:zariskisetforonemap}, for each $d_i$ we can find a Zariski dense subset $\cH(g,d_i,K)$ of $K^d$ such that for all monic polynomials $f(x)$ with critical points $b_1,\dots, b_{d_i-1}$ and constant term $c_0$ with $(b_1,\dots,b_{d_i-1},c_0)\in \cH(g,d_i,K)$, 
the density of the set of primes $\fp$ of $K$ for which $f(x)$ has a finite attracting periodic point in $\bar{K}_\fp$ is less than $\epsilon$.  

Suppose $\Phi=(f_1,\dots,f_g)$, where $f_i(x)$ is a monic polynomial with critical points $b_{i,1},\dots,b_{i,d-1}$ and constant term $c_{i,0}$ and $(b_{i,1},\dots,b_{i,d-1},c_0)\in \cH(g,d_i,K)$.
Let $S_i$ be the set of primes for which $f_i$ has an attracting periodic point in $K_\fp$, then $\delta(S_i)<\epsilon$. Let $S$ be the union of the $S_i$ and the finitely many primes for which some $f_i$ has bad reduction, then $\delta(S)\leq \sum_g \delta(S_i)<g\epsilon=1$. Thus, the set of primes for which each $f_i$ has good reduction and none of the $f_i$ have an attracting periodic point in $\bar{K_\fp}$ has positive density and so contains at least one prime. The result now follows from Theorem \ref{thm:avoidingramification}.
\end{proof}

\newcommand{\etalchar}[1]{$^{#1}$}
\providecommand{\bysame}{\leavevmode\hbox to3em{\hrulefill}\thinspace}
\providecommand{\MR}{\relax\ifhmode\unskip\space\fi MR }
\providecommand{\MRhref}[2]{%
  \href{http://www.ams.org/mathscinet-getitem?mr=#1}{#2}
}
\providecommand{\href}[2]{#2}

\end{document}